\documentclass[11pt, reqno]{amsart}
\usepackage[utf8]{inputenc}
\numberwithin{equation}{section}
\usepackage{color}
\usepackage[shortlabels]{enumitem}
\usepackage[hidelinks]{hyperref}
\usepackage{cleveref}
\usepackage{todonotes}
\newcommand{\qtq}[1]{\quad\text{#1}\quad}
\usepackage{mathabx}
\usepackage{amsmath, amssymb, amsfonts, amsthm, mathtools}
\usepackage{dsfont}
\usepackage[scr=boondoxo]{mathalfa}

\theoremstyle{definition}
\newtheorem{definition}{Definition}

\theoremstyle{plain}
\newtheorem{theorem}[definition]{Theorem}
\newtheorem*{theorem*}{Theorem}
\newtheorem{lemma}{Lemma}

\theoremstyle{remark}
\newtheorem*{claim*}{Claim}

\newtheorem{remark}[]{Remark}
\newtheorem*{remark*}{Remark}

\newtheorem*{example*}{Example}

\newcommand{\eps}{\varepsilon}

\DeclareMathOperator{\R}{\mathbb{R}}

\DeclareMathOperator{\C}{\mathbb{C}}

\DeclareMathOperator{\bigO}{\mathcal{O}}
\DeclareMathOperator{\F}{\mathcal{F}}

\DeclareMathOperator{\M}{\mathcal{M}}

\newcommand{\sW}{\mathscr{W}}

\DeclareMathOperator{\Sw}{\mathcal{S}}

\newcommand{\jbrak}[1]{\langle#1\rangle}

\renewcommand{\d}{\mathrm{d}}
\renewcommand{\:}{\colon}

\newcommand{\bbar}{\overline}

\renewcommand{\tilde}{\widetilde}
\renewcommand{\hat}{\widehat}

\renewcommand{\d}{\mathrm{d}}
\newcommand{\dxi}{\, \mathrm{d}\xi}
\newcommand{\dx}{\, \mathrm{d}x}
\newcommand{\dy}{\, \mathrm{d}y}
\newcommand{\ds}{\, \mathrm{d}s}

\newcommand{\dz}{\, \d z}

\newcommand{\cR}{\mathcal{R}}

\newcommand{\cD}{\mathcal{D}}

\newcommand{\defe}{\overset{\mathrm{def}}{=}}

\begin{document}
    
    \title[Modified Scattering for Hartree NLS]{Modified Scattering for the Hartree Nonlinear Schr\"odinger Equation}
    \author[T. Van Hoose]{Tim Van Hoose}
    \address{Department of Mathematics, University of North Carolina, Chapel Hill}
    \email{tvh@unc.edu}
    \begin{abstract}
        We prove sharp $L^\infty$ decay and modified scattering for the Hartree nonlinear Schr\"odinger equation in dimensions $2$ and $3$ using the testing by wavepackets method of Ifrim and Tataru \cite{ifrimGlobalBoundsCubic2014}. We show that the scattering behavior happens at a regularity well below that of earlier results such as \cite{hayashiAsymptoticsLargeTime1998a,katoNewProofLong2010a}.
    \end{abstract}
    \maketitle

    \section{Introduction}
        We will consider the long-time behavior for small solutions to the following Hartree nonlinear Schr\"odinger equation:
        \begin{equation}\label{E:VNLS}
            \begin{cases}
                i\partial_t u + \Delta u = (|\cdot|^{-1} \ast |u|^2) u \\
                u(0, x) = u_0(x) 
            \end{cases}\quad \text{on } \R_t \times \R_x^d.
        \end{equation}
        where $d \in \{2, 3\}$. Our primary goal is to prove a modified scattering result using the testing by wavepackets technique of Ifrim and Tataru \cite{ifrimGlobalBoundsCubic2014,ifrimTestingWavePackets2022a}. As a consequence, we obtain modified scattering well below the `classical' regularity of $H^{\frac{d}{2}+, \frac{d}{2}+}$; our result works for data in the class $H^{0, \frac{d}{2}+}$. Here, the space $H^{\gamma, \nu}(\R^d)$ is a weighted $L^2$-Sobolev space defined by the norm \eqref{E:wgtSobnorm}, and by $\frac{d}{2}+$ we mean $\frac{d}{2}+\eps$ for any $\eps > 0$.
        
        Our main theorem reads as follows: 
        \begin{theorem}\label{T:maintheorem}
            Let $d \in \{2, 3\}$, $0 < \eps \ll 1$, and $\|u_0\|_{H^{0, \beta}(\R^d)} =\eps$. Then there exists a unique global solution $u(t,x)$ belonging to $H^{0, \beta}$ in the sense that $e^{-it\Delta}u \in L_t^\infty H_x^{0, \beta}(\R \times \R^d)$. Furthermore the following additional properties hold for the solution $u$: 
            \begin{enumerate}
                \item Sharp $L^\infty$ decay and energy growth:
                    The solution $u$ satisfies the estimates 
                    \begin{equation} \label{E:sharpdecay}
                        \|u\|_{L^\infty} \lesssim 2\eps |t|^{-\frac{d}{2}}
                    \end{equation}
                    and
                    \begin{equation}\label{E:energygrowth}
                        \|e^{-it\Delta}u\|_{H^{0, \beta}} \lesssim 2\eps\jbrak{t}^{C\eps^{3-\frac{1}{d}}}.
                    \end{equation}
                \item Modified scattering: If $u$ is a solution to \eqref{E:VNLS}, then there exists a profile $\sW \in L^\infty(\R^d)$ so that (in the $L^\infty$ topology)
                \begin{equation}
                    u(t, x) = t^{-\frac{d}{2}}e^{i\frac{|x|^2}{4t}}\sW\left(\frac{x}{2t}\right)e^{-\frac{i}{2}\log(t)(|\cdot|^{-1} \ast |\sW(\frac{x}{2t})|^2)} + \mathcal{O}(t^{-\frac{d}{2}-\eps}) \text{ as } t \to \infty.
                \end{equation}
            \end{enumerate}
        \end{theorem}
        Our main theorem fits within the context of modified scattering for Schr\"odinger equations with long-range nonlinearities. Here, `modified scattering' refers to the fact that the asymptotic behavior appearing in \Cref{T:maintheorem} is not that of linear scattering, due to the presence of the logarithmic correction in the asymptotic expansion of $u(t, x)$. This logarithmic correction comes from the fact that the Hartree nonlinearity is `scattering-critical': heuristically, one can check that we have the estimate 
        \begin{equation*}
            \| |\cdot|^{-1} \ast |e^{it\Delta} \varphi|^2\|_{L^\infty} = \mathcal{O}(t^{-1}), \qtq{for any} \varphi \in \Sw(\R^d).
        \end{equation*}
        This $\mathcal{O}(t^{-1})$ time decay is on the borderline between integrability and nonintegrability at infinity, and is directly responsible for the appearance of the logarithmic phase correction. \par
        Modified scattering for various types of Schr\"odinger equations has been studied in many different works \cite{deiftLongtimeAsymptoticsSolutions2002, hayashiAsymptoticsLargeTime1998a, katoNewProofLong2010a, lindbladScatteringSmallData2006, ifrimGlobalBoundsCubic2014, ifrimTestingWavePackets2022a, ozawaLongRangeScattering1991, ginibreLongRangeScattering1993a}. We will give a brief overview comparing the features of the different methods; for a high-quality exposition in the context of the $1d$ cubic NLS, we refer to the excellent paper \cite{murphyReviewModifiedScattering2021}. \par
        The first method closely related to our work is that of Hayashi and Naumkin \cite{hayashiAsymptoticsLargeTime1998a}. To prove modified scattering, they formulate a bootstrap argument involving a `dispersive' norm and an `energy' norm. The goal is to prove that the dispersive norm remains small (this will establish the $L^\infty$ decay estimate), while the energy norm is allowed to grow a little bit. Once the bootstrap is closed, one can solve an ODE for the `profile' $f(t) = e^{-it\Delta}u(t)$ to derive the claimed asymptotic expansion. \par
        A similar work in the spirit of \cite{hayashiAsymptoticsLargeTime1998a} is that of Kato and Pusateri in \cite{katoNewProofLong2010a}. They also formulate a bootstrap argument involving a dispersive-type norm and an energy-type norm. However, their method of proving estimates takes place mostly in Fourier space, and involves the detailed analysis of various oscillatory integrals using the space-time resonance method of Germain, Masmoudi and Shatah \cite{germainGlobalSolutions3D2009}. In the end, one arrives at a similar ODE to that found in \cite{hayashiAsymptoticsLargeTime1998a}, after which point the argument proceeds along the same lines.\par
        In the paper \cite{lindbladScatteringSmallData2006}, the authors take a slightly different approach to proving modified scattering. They consider the associated function $q(t, x) = e^{i|x|^2/4t}u(t, \frac{x}{2t})$. 
         If we consider the PDE satisfied by $q(t)$, using an integrating factor similar to the previous two methods yields an estimate of the form 
        \begin{equation*}
            \|q(t)\|_{L^\infty} \lesssim \|q(0)\|_{L^\infty} + \int_1^t \frac{1}{2s^2}\|\partial_x^2 q(s)\|_{L^\infty} \ds.
        \end{equation*}
        If $q$ is assumed to have higher regularity, then we can use energy estimates to control the $L^\infty$ norm under the integral sign using interpolation estimates. One can then close a bootstrap estimate and from there conclude the same asymptotic behavior as before. \par

        Finally, we turn to the testing by wavepackets method of Ifrim and Tataru \cite{ifrimGlobalBoundsCubic2014, ifrimTestingWavePackets2022a}. This method proceeds by testing a solution $u(t, x)$ against a Schwartz function localized at the scale $t^{1/2}$ dictated by the Heisenberg uncertainty principle for the Schr\"odinger equation. Specifically, we define 
        \[
            \Psi_v(t, x) =  e^{i\frac{|x|^2}{4t}}\vartheta\left(\frac{x-2tv}{t^{\frac{1}{2}}}\right)
        \]
        and set $\gamma(t, v) =\jbrak{u, \Psi_v}_{L^2}$, where $\vartheta$ is Schwartz with integral $1$. We call $\gamma(t, v)$ a \textit{wavepacket} (with velocity $v$). One can show that $\Psi_v$ is an approximate solution to the linear Schr\"odinger equation, in the sense that $(i\partial_t + \Delta)\Psi_v$ solves the equation up to an error of size $\mathcal{O}(t^{-1})$. The argument proceeds by first proving that \eqref{E:VNLS} is globally well-posed in a suitable sense (\Cref{T:GWP}). We then prove a series of bounds comparing wavepackets to the solution $u(t, x)$ along a particular ray parametrized by velocity vectors $v$ in \Cref{L:gammabounds}. 

        Finally, we derive an approximate ODE satisfied by $\gamma(t, v)$. This is essentially the same ODE that has been derived in all previous works. The goal is then to show that the error in the approximation is integrable in time. Once we've done that, we can use the ODE and the global existence theory to close a bootstrap argument. This bootstrap will prove that $\|u(t, x)\|_{L^\infty}$ decays like $t^{-\frac{d}{2}}$. With the bootstrap closed, we can use a similar argument to \cite{hayashiAsymptoticsLargeTime1998a} and \cite{katoNewProofLong2010a} to write down the explicit asymptotic behavior for the solution $u(t, x)$. 
        One of the main features of this method is that it markedly improves upon the regularity requirements of the data and solution. Indeed, the methods mentioned previously all require the initial datum to lie in (essentially) the weighted Sobolev space $H^{\frac{d}{2}+, \frac{d}{2}+}(\R^d)$, which allows one to use the $L^\infty$ Sobolev embedding estimate. The testing by wavepackets method will allow us to improve this result substantially, allowing us to treat data in the space $H^{0, \frac{d}{2}+}(\R^d)$. 

        For a different perspective on the testing-by-wavepackets method, we also  direct the reader to another result of \cite{cloos2020long}, which applies testing-by-wavepackets in the setting of the Dirac-Maxwell equation, a different (nonlocal) model based on the Klein-Gordon equation rather than the Schr\"odinger equation.
        
        We organize the remainder of the paper as follows: in \Cref{S:notation}, we compile some notation, lemmas and definitions for quantities that will be used in the remainder of the paper. In \Cref{S:GWP}, we will prove the global well-posedness result we claimed earlier and the bounds along the ray used to prove the first part of \Cref{T:maintheorem}. Finally, in \Cref{S:asymptotics}, we will use the results of \Cref{S:GWP} to write down the asymptotic expansion appearing as the second part of \Cref{T:maintheorem}.
        \subsection*{Acknowledgements} We are thankful to Jeremy Marzuola for his guidance and support, for many helpful conversations about the problem, and his careful reading of the paper. We also thank Benjamin Harrop-Griffiths and Sebastian Herr for their review of the draft and their input. T.V.H. was supported by the UNC NSF DMS-2135998 Research Training Grant.
        
    \section{Notation}\label{S:notation}
    In this section, we introduce some notation that will be used throughout the remainder of the paper.\par
    We write $A \lesssim B$ or $B \gtrsim A$ to denote the inequality $A \leq CB$ for some constant $C > 0$, where $C$ may depend on parameters like the dimension or the indices of function spaces. If $A \lesssim B $ and $B \lesssim A$ both hold, then we write $A \sim B$. We will also make use of the standard Landau symbol $\bigO$, as well as the Japanese bracket notation $\jbrak{\cdot} := (1+|\cdot|^2)^{\frac{1}{2}}$. \par
    We use the standard Lorentz spaces $L^{p,q}(\R^d)$, specifically the fact that $|\cdot|^{-1} \in L^{d, \infty}(\R^d)$. These spaces are defined by the (quasi)norm 
    \begin{equation*}
        \|f\|_{L^{p,q}(\R^d)} :=
        \begin{cases}
                   p^{\frac{1}{q}} \left(\displaystyle\int_0^\infty t^q \ m\{x \in \R^d \: |f(x)| \geq \lambda\}^{\frac{q}{p}}\frac{\d \lambda} {\lambda}\right)^{\frac{1}{q}}, q< \infty \\
                   \displaystyle\sup_{\lambda >0} \lambda^p \ m\{x \in \R^d \: |f(x)| \geq \lambda\}, q = \infty
        \end{cases}     
    \end{equation*}
    where $m$ is $d$-dimensional Lebesgue measure.

    
    We will denote the Fourier transform by $\F[f](\xi) = \widehat{f}(\xi)$ with the normalization 
    \begin{equation*}
        \F[f](\xi) := (2\pi)^{-\frac{d}{2}} \int_{\R^d} e^{-ix \cdot\xi} f(x) \dx 
    \end{equation*}
    and inverse 
    \begin{equation}
        \F^{-1}[g](x) = \check{g}(x) := (2\pi)^{-\frac{d}{2}} \int_{\R^d} e^{ix\cdot\xi} g(\xi)\d \xi.
    \end{equation}
    We will define the weighted Sobolev spaces $H^{\gamma, \nu}(\R^d)$ by the norm
    \begin{equation}\label{E:wgtSobnorm}
        \|u\|_{H^{\gamma, \nu}} := \|\jbrak{\nabla}^\gamma u\|_{L^2} + \| |x|^\nu u\|_{L^2}
    \end{equation}
    where as usual $\jbrak{\nabla}^\gamma := \F^{-1} \jbrak{\xi}^\gamma \F$, and we define the standard Sobolev spaces $H^s(\R^d):= H^{s, 0}(\R^d)$ in the notation above. \par
    In the usual way, we will denote the free Schr\"odinger propagator by $e^{it\Delta}$. Direct computation shows that we can decompose this operator as 
    \begin{equation}\label{E:MDFM}
        e^{it\Delta} = \M(t) \cD(t) \F \M(t)
    \end{equation}
    where 
    \begin{equation}
        \M(t)f(x) = e^{i\frac{|x|^2}{4t}} f(x) \qquad \text{and} \qquad \cD(t) = (2it)^{-\frac{d}{2}} f\left(\frac{x}{2t}\right).
    \end{equation}
    By direct computation, we see that 
    \begin{equation}
        \cD(t)^{-1} = (2i)^{d} \cD\left(\frac{1}{t}\right) .
    \end{equation}
    We will also make use of the Galilean operator $J(t):= x + 2it\nabla$. Direct computation shows that
    \begin{equation}\label{E:Jtdefn1}
        J(t) = \M(t) (2it\nabla) \M(-t).
    \end{equation}
    Indeed, we have 
    \begin{align*}
        \M(t)(2it\nabla)\M(-t) f &= e^{i\frac{|x|^2}{4t}}(2it\nabla)e^{-i\frac{|x|^2}{4t}}f \\
        &= e^{i\frac{|x|^2}{4t}}e^{-i\frac{|x|^2}{4t}} (x+2it\nabla)f \\
        &= J(t)f.
    \end{align*}
    An ODE argument furnishes the identity 
    \begin{equation}\label{E:Jtdefn2}
        J(t) = e^{it\Delta}xe^{-it\Delta}.
    \end{equation}
    Indeed, both sides of the equation match at $t=0$. If we take a time derivative on both sides, we find that 
    \begin{align*}
        J'(t) &= 2i\nabla \\
        \frac{d}{dt} [e^{it\Delta} x e^{-it\Delta}] &= i e^{it\Delta} [\Delta, x] e^{-it\Delta}
    \end{align*}
    where $[\cdot, \cdot]$ is the usual operator commutator. One can check directly that $[\Delta, x] = 2\nabla$. Since Fourier multipliers commute and $(e^{it\Delta})_{t \in \R}$ is a semigroup, we see 
    \begin{equation*}
        i e^{it\Delta} [\Delta, x] e^{-it\Delta} = 2i\nabla.
    \end{equation*} 
    Since the time derivatives match for all $t$ and the two expressions have matching values at $t=0$, ODE uniqueness implies that $J(t) = e^{it\Delta}x e^{-it\Delta}$ for all $t$, which was what we claimed. 
    
    \begin{remark*}
        Since $e^{it\Delta}: L^2 \to L^2$ is unitary, we see that $ \| J(t)u\|_{L^2} = \|x e^{-it\Delta}u\|_{L^2}.$
    \end{remark*}
    We also define powers of $J(t)$ in the following manner:
    \begin{align}
        |J|^{\gamma}(t) &:= \M(t)(-4t^2 \Delta)^{\frac{\gamma}{2}}\M(-t) \quad \text{for }\gamma\in[0, \infty) \label{E:Jtpower}\\
        &= e^{it\Delta}|x|^\gamma e^{-it\Delta}. \label{E:Jtpower2}
    \end{align}
    Finally, we will need an abstract interpolation result from \cite{berghInterpolationSpacesIntroduction1976}, which we have specialized to our particular case. This will allow us to directly estimate convolutions against $|x|^{-1}$ in $L^\infty$, which is an inadmissible endpoint for Hardy-Littlewood-Sobolev. 
    
    \begin{lemma}[Real Interpolation of $L^p$ spaces]\label{L:interplemma}
        If $1 \leq p_0 \neq p_1 \leq \infty$, then we have the equality of spaces (with equivalent norms) 
        \begin{equation}
            [L^{p_0}, L^{p_1}]_{\theta, q} = L^{p, q},
        \end{equation}
        where $\frac{1}{p} = \frac{1-\theta}{p_0} + \frac{\theta}{p_1}$ and $0 < \theta < 1$. Further, it holds that 
        \begin{equation}
            \|u\|_{L^{p, q}} \lesssim \|u\|_{L^{p_0}}^{1-\theta} \|u\|_{L^{p_1}}^\theta.
        \end{equation}
    \end{lemma}
    \begin{remark}
        The first part of this lemma is directly from \cite[Theorem 5.3.1]{berghInterpolationSpacesIntroduction1976}. The second half follows from the discussion in \cite[Section 3.5.1f]{berghInterpolationSpacesIntroduction1976}. In our particular instance, this lemma will be applied to the spaces $L^2$ and $L^\infty$. To be completely explicit, we will use that 
        \begin{equation*}
            \|u\|_{L^{\frac{2d}{d-1}, 2}} \lesssim \|u\|_{L^2}^{\frac{d-1}{d}} \|u\|_{L^\infty}^{\frac{1}{d}},
        \end{equation*}
        since $\frac{2d}{d-1} > 2$, and the exponents come directly from \Cref{L:interplemma}. For more explicit details, see \Cref{AS:InterpAppendix}.
    \end{remark}
    \begin{remark}[Convention for convolutions]\label{R:convolution}
        We will frequently work with the operation of convolution with the kernel $|\cdot|^{-1}$. For notational ease, we will write $|\cdot|^{-1} \ast f(t, z)$ to mean the convolution with output variable $z$: 
    \begin{equation*}
        |\cdot|^{-1} \ast f(t,z) := \int_{\R^d} \frac{f(t,y)}{z-y} \dy.
    \end{equation*}
    For example, we will often see convolution with functions of the variable $2tv$. We interpret this in light of the definition above, with the replacement $z \mapsto 2tv$. If the argument is suppressed, we take that to mean that the output is the variable $x$. 
    \end{remark}
    \section{Global Well-Posedness and Estimates along Rays}\label{S:GWP}
        In this section, we will prove the claimed global well-posedness result, along with the estimates along rays that will later prove the first part of \Cref{T:maintheorem}.
        To begin, we prove that \eqref{E:VNLS} is globally well-posed in $H^{0, \beta}$; we will do this by a standard $L^2$ local well-posedness result combined with a persistence of regularity argument.
        \begin{theorem}\label{T:GWP}
            Let $d \in \{2, 3\}$. Then the equation \eqref{E:VNLS} is globally well-posed in $H^{0, \beta}$ for any $\beta > \frac{d}{2}$, in the sense that it has a unique solution $u \in C_t L_x^2(\R \times \R^d)$ with $|J|^\beta u \in C_tL_x^2(\R \times \R^d)$. Any such solution $u(t, x)$ is in $C_t L_x^\infty$, and near $t = 0$ we have 
            \begin{equation}
                \|u(t, x)\|_{L_x^\infty} \lesssim t^{-\frac{d}{2}} \|u_0\|_{H^{0, \beta}}.
            \end{equation}
        \end{theorem}
        \begin{proof}
            To prove the $L^2$ well-posedness, we rephrase the problem using the Duhamel formula, leading to the equation 
            \begin{equation}
                \Phi[u] := u(t,x) = e^{it\Delta}u_0 -i \int_0^t e^{i(t-s)\Delta}\{(|x|^{-1} \ast |u|^2)u(s)\} \ds.
            \end{equation}
            We have to treat the cases $d = \{2, 3\}$ with different Strichartz norms, but the arguments are entirely identical. To this end, define 
            \begin{equation*}
                A_2 = L_{T,x}^4(\R^2) \qtq{and} A_3 = L_T^5 L_x^{\frac{30}{11}}(\R^3).
            \end{equation*}
            We will run a contraction mapping argument in the space 
            \begin{equation}
                X_{T, d} = \left\{u\: [0, T] \times \R^d \to \C \mid \|u\|_{L_T^\infty L_x^2} \lesssim 2\|u_0\|_{L^2} \text{ and } \|u\|_{A_d}\lesssim 2C\|u_0\|_{L^2}\right\}
            \end{equation}
            endowed with the metric $d(u, v) = \|u-v\|_{L_T^\infty L_x^2}$, and where $C$ encodes all the arbitrary constants that come from Strichartz estimates and Hardy-Littlewood-Sobolev. For the sake of brevity, we present only the key nonlinear estimates here:
            \begin{align*}
                \left\| \int_0^t e^{i(t-s)\Delta}(|x|^{-1} \ast |u|^2)u(s)\ds\right\|_{L_T^\infty L_x^2(\R^2)} &\lesssim \| (|\cdot|^{-1} \ast |u|^2)u(s)\|_{L_{T, x}^{\frac{4}{3}}} \\
                &\lesssim \|u\|_{L_T^\infty L_x^2} \| |\cdot|^{-1} \ast |u|^2\|_{L_T^{\frac{4}{3}}L_x^4} \\
                &\lesssim \|u\|_{L_T^\infty L_x^2} \||u|^2\|_{L_{T, x}^{\frac{4}{3}}} \\
                &\lesssim T^{\frac{1}{2}}\|u\|_{L_T^\infty L_x^2}^2 \|u\|_{A_2},
            \end{align*}
            which is acceptable. The estimate of the $A_2$ component of the norm is entirely identical, by Strichartz. In three spatial dimensions, we have the nonlinear estimate
            \begin{align*}
                \left\| \int_0^t e^{i(t-s)\Delta} (|x|^{-1} \ast |u|^2)u(s)\ds\right\|_{L_t^\infty L_x^2(\R^3)} &\lesssim \| (|x|^{-1} \ast |u|^2)u\|_{L_{T, x}^{\frac{10}{7}}}\\
                &\lesssim \|u\|_{L_T^\infty L_x^2}\| |x|^{-1} \ast |u|^2\|_{L_T^{\frac{10}{7}}L_x^5} \\
                &\lesssim \|u\|_{L_T^\infty L_x^2}\||u|^2\|_{L_T^{\frac{10}{7} }L_x^{\frac{15}{13}}}\\
                &\lesssim T^{\frac{1}{2}}\|u\|_{L_T^\infty L_x^2}^2 \|u\|_{A_3},
            \end{align*}
            where to go from the second to the third line we applied the Hardy-Littlewood-Sobolev inequality. Using the definition of the space $X_{T, d}$, we see that for $T$ small enough depending only on $\|u_0\|_{L^2}$, we have a map from $X_{T, d} \to X_{T, d}$ defined by the right-hand side of the Duhamel formula above. To prove that $\Phi$ is a contraction map, we consider the difference $\|\Phi[u]-\Phi[v]\|_{L_T^\infty L_x^2}$. To estimate this, we see that once again the crucial term is the integral term of the Duhamel formula. To estimate it, we rewrite the difference of convolutions as 
            \begin{multline}
                (|\cdot|^{-1} \ast |u|^2)u - (|\cdot|^{-1} \ast |v|^2)v = (|\cdot|^{-1} \ast |u|^2)(u-v) \\+ (|\cdot|^{-1} \ast (u-v)\bbar{u})v+ (|\cdot|^{-1} \ast (\bbar{u}- \bbar{v})v)v.
            \end{multline}
            Checking that $\Phi$ is a contraction mapping for $T$ sufficiently small proceeds in an entirely analogous way to the estimate above. Hence there is a unique $L^2$ solution to \eqref{E:VNLS}; by conservation of mass, it is global. \par
            We would now like to upgrade to $H^{0, \beta}$ initial data; this will follow from a standard persistence of regularity argument. To do this, we note that the operator $|J|^\beta$ defined above commutes nicely with the linear part of the equation. By Strichartz, the linear term is easily controlled in terms of $\|u_0\|_{H^{0, \beta}}$; it remains to handle the nonlinear term, which takes the form 
            \begin{equation}\label{E:persistenceintegral}
                \M(t) (2it)^\beta |\nabla|^\beta \M(-t) \int_0^t e^{i(t-s)\Delta} (|\cdot|^{-1} \ast |u|^2)u \ds.
            \end{equation}
            Estimating this in $L_T^\infty L_x^2$ and applying Strichartz, we see that it suffices to provide good estimates for 
            \begin{equation}
               \| |\nabla|^\beta [(|\cdot|^{-1} \ast |w|^2)w]\|_{L_{T, x}^{\frac{4}{3}}} \qtq{and} \| |\nabla|^{\beta} [(|\cdot|^{-1} \ast |w|^2)w]\|_{L_{T, x}^{\frac{10}{7}}},
            \end{equation}
            where we write $w = \M(-t)u$. By the fractional chain rule, we can control the terms above by
            \begin{align*}
                \| |\nabla|^\beta [(|\cdot|^{-1} \ast |w|^2)w]\|_{L_{T, x}^{\frac{4}{3}}} &\lesssim
                \begin{multlined}[t]
                    \|w\|_{L_T^\infty L_x^2} \| |\cdot|^{-1} \ast |\nabla|^\beta |w|^2\|_{L_T^\frac{4}{3}L_x^4} \\
                    + \||\nabla|^\beta w \|_{L_T^\infty L_x^2} \| |\cdot|^{-1} \ast |w|^2\|_{L_T^{\frac{4}{3}}L_x^4}
                \end{multlined}\\
                &\lesssim T^{\frac{1}{2}} \|u\|_{L_T^\infty L_x^2} \|u\|_{A_2} \| |\nabla|^\beta w\|_{L_T^\infty L_x^2}
            \end{align*}
            in two spatial dimensions, and in three spatial dimensions, we have 
            \begin{align*}
                \| |\nabla|^{\beta} [(|\cdot|^{-1} \ast |w|^2)w]\|_{L_{T, x}^{\frac{10}{7}}} &\lesssim \begin{multlined}[t]
                    \|u\|_{L_T^\infty L_x^2} \| |\cdot|^{-1} \ast |\nabla|^\beta |u|^2\|_{L_T^{\frac{10}{7}} L_x^5} \\
                    +\| |\cdot|^{-1} \ast |u|^2 \|_{L_T^{\frac{10}{7}}L_x^5}\| |\nabla|^{\beta}u\|_{L_T^\infty L_x^2}.
                \end{multlined}\\
                &\lesssim T^{\frac{1}{2}}\|u\|_{L_T^\infty L_x^2} \|u\|_{A_3}\||\nabla|^{\beta}w\|_{L_T^\infty L_x^2}.
            \end{align*}
            Reinserting this into \eqref{E:persistenceintegral} and recalling the definition of $|J|^\beta$, we see that 
            \begin{equation}
                \| |J|^{\beta}u\|_{L_T^\infty L_x^2} \lesssim \|u_0\|_{H^{0, \beta}} + \frac{1}{2}\||J|^\beta u\|_{L_T^\infty L_x^2},
            \end{equation}
            where by the contraction mapping argument from earlier, we can choose $T$ depending only on $\|u_0\|_{L^2}$ to make the constant in front of $|J|^\beta u$ on the right-hand side equal to $\frac{1}{2}$. This implies that on the interval $[0, T]$ the norm of $|J|^\beta u$ grows by no more than a factor of $2$ in terms of the initial data. Again invoking mass conservation, we see that $\||J|^\beta u\|_{L^2}$ is finite along the global flow. \par
            Finally, we would like to obtain $L^\infty$ bounds for the solution. This follows from a virtually identical argument to \cite{ifrimGlobalBoundsCubic2014}; we reproduce it in our case for posterity. First, note (with $w$ defined the same way as earlier) that we have the identity
            \begin{equation}\label{E:jbetaidentity}
                \M(-t) |J|^\beta u = (2it)^\beta |\nabla|^\beta w.
            \end{equation}
            In particular, we see immediately that $w \in H^\beta$, and thus in $L^\infty$; to conclude, we use Gagliardo-Nirenberg:
            \begin{equation}
                \|u(t)\|_{L^\infty} = \|w(t)\|_{L^\infty} \lesssim \|w(t)\|_{L^2}^{1-\frac{d}{2\beta}}\||\nabla|^\beta w\|_{L^2}^{\frac{d}{2\beta}}. 
            \end{equation}
            The right-hand side is easily seen to be bounded by $t^{-\frac{d}{2}}\|u_0\|_{H^{0, \beta}}$ by rearranging \eqref{E:jbetaidentity} and using the properties of the solution $u$.        
        \end{proof}
        Motivated by the fact that linear Schr\"odinger waves propagate with velocity $v = \frac{x}{2t}$, given a velocity vector $v \in \R^d$, we define a ray 
        \begin{equation}
            \Gamma_{v} := \{x = 2t v\}.
        \end{equation}
        Following \cite{ifrimGlobalBoundsCubic2014}, we then make the following
        \begin{definition}
            Let $\vartheta \in \Sw(\R^d)$ have total integral $1$. Then a wave packet with velocity $v$ is defined to be 
            \begin{equation*}
                \Psi_{v}(t, x) := \vartheta\left(\frac{x -2tv}{t^{\frac{1}{2}}}\right)e^{i\frac{|x|^2}{4t}}.
            \end{equation*}
        \end{definition}
        The reason for the rescaling $t^{-\frac{1}{2}}$ of the argument of $\vartheta$ is so that the wavepacket lives at the scaling dictated by the Heisenberg uncertainty principle. \par
        In fact, for any $v \in \R^d$, $\Psi_v(t,x)$ is an approximate solution to the linear Schr\"odinger equation in the following sense:
        \begin{lemma}
            In $d$ dimensions, we have for any $v \in \R^d$
            \begin{equation}\label{E:approxsoln}
                (i\partial_t + \Delta)\Psi_v(t, x) = \frac{1}{2t}e^{i\frac{|x|^2}{4t}} \nabla \cdot \bigg\{
                \begin{multlined}[t]
                    i(x-2vt)\vartheta\left(\frac{x-2vt}{t^\frac{1}{2}}\right) \\+ 2t^{\frac{1}{2}}\nabla\vartheta\left(\frac{x-2vt}{t^\frac{1}{2}}\right)\bigg\}
                \end{multlined}
            \end{equation}
        \end{lemma}
        \begin{proof}
            This follows by a direct (albeit involved) computation. To wit, we have 
            \begin{equation*}
                i\partial_t \Psi_v(t, x) = \frac{|x|^2}{4t^2}e^{i\frac{|x|^2}{4t}} \vartheta\left(\frac{x -2tv}{t^{\frac{1}{2}}}\right) + i\left(-\frac{1}{2}xt^{-\frac{3}{2}} - \frac{2vt^{-\frac{1}{2}}}{2}\right)\cdot (\nabla\vartheta)\left(\frac{x -2tv}{t^{\frac{1}{2}}}\right)e^{i\frac{|x|^2}{4t}}.
            \end{equation*}
            Next, we need to handle the spatial derivatives using 
            \[
                \Delta(fg) = f \Delta g + 2(\nabla f) \cdot (\nabla g) + g \Delta f.    
            \]
            Direct computation yields 
            \begin{equation*}
                \nabla e^{i\frac{|x|^2}{4t}} = \frac{ix}{2t}e^{i\frac{|x|^2}{4t}}
            \end{equation*}
            and 
            \begin{equation*}
                \nabla \vartheta\left(\frac{x -2tv}{t^{\frac{1}{2}}}\right) = \frac{1}{t^\frac{1}{2}}(\nabla\vartheta)\left(\frac{x -2tv}{t^{\frac{1}{2}}}\right). 
            \end{equation*}
            To compute the second derivatives, we use the product rule for gradients and divergences. This yields
            \begin{equation}
                \Delta e^{i\frac{|x|^2}{4t}} = \frac{id}{2t}e^{i\frac{|x|^2}{4t}} - \frac{|x|^2}{4t^2} e^{i \frac{|x|^2}{4t}}
            \end{equation}
            and 
            \begin{equation}
                \Delta \vartheta\left(\frac{x-2tv}{t^\frac{1}{2}}\right) = \frac{1}{t} (\Delta \vartheta)\left(\frac{x-2tv}{t^\frac{1}{2}}\right).
            \end{equation}
            Putting everything together, we have
            \begin{equation}\label{E:wavepacketpde1}
                \begin{aligned}
                    (i\partial_t + \Delta) \Psi_v (t, x) &= 
                    \frac{|x|^2}{4t^2}e^{i\frac{|x|^2}{4t}} \vartheta\left(\frac{x -2tv}{t^{\frac{1}{2}}}\right) \\
                    &+ i\left(-\frac{1}{2}xt^{-\frac{3}{2}} - \frac{2vt^{-\frac{1}{2}}}{2}\right)\cdot (\nabla\vartheta)\left(\frac{x -2tv}{t^{\frac{1}{2}}}\right)e^{i\frac{|x|^2}{4t}} \\
                    &+ \left(\frac{id}{2t}e^{i\frac{|x|^2}{4t}} - \frac{|x|^2}{4t^2} e^{i \frac{|x|^2}{4t}}\right)\vartheta\left(\frac{x -2tv}{t^{\frac{1}{2}}}\right) \\ 
                    &+ \frac{ix}{t^\frac{3}{2}} e^{i\frac{|x|^2}{4t}} \cdot (\nabla \vartheta)\left(\frac{x -2tv}{t^{\frac{1}{2}}}\right) \\
                    &+ \frac{1}{t} (\Delta\vartheta)\left(\frac{x -2tv}{t^{\frac{1}{2}}}\right) e^{i \frac{|x|^2}{4t}},
                \end{aligned}
            \end{equation}
            where the first line contains the time derivative and the second contains the space derivatives. One may simplify this expression somewhat with some simple algebra: 
            \begin{equation}
                \mathrm{RHS}\eqref{E:wavepacketpde1} =
                \begin{multlined}[t]
                    i\left(\frac{x-2vt}{2t^{\frac{3}{2}}}\right)\cdot (\nabla\vartheta)\left(\frac{x -2tv}{t^{\frac{1}{2}}}\right)e^{i\frac{|x|^2}{4t}} \\+\frac{id}{2t}e^{i\frac{|x|^2}{4t}}\vartheta\left(\frac{x -2tv}{t^{\frac{1}{2}}}\right) 
                    + \frac{1}{t}(\Delta \vartheta) \left(\frac{x-2vt}{t^\frac{1}{2}}\right)e^{i\frac{|x|^2}{4t}}.
                \end{multlined} 
            \end{equation}
        \end{proof}
        The point of this lemma is that the piece of \eqref{E:approxsoln} in brackets has exactly the same localization properties as $\vartheta$, but decays a factor of $\tfrac{1}{t}$ better than the original wavepacket. \par
        Next, define 
        \begin{equation*}
            \gamma(t, v) = \int_{\R^d} u(t, x) \bbar{\Psi_v}(t, x) \dx.
        \end{equation*}
        This quantity measures the decay of the solution $u$ along the ray $\Gamma_v$. Using Plancherel's theorem we can rewrite this expression as 
        \begin{equation*}
            \gamma(t, v) = \int_{\R^d} \hat{u}(t, \xi) \bbar{\hat{\Psi_v}}(t, \xi) \dxi.
        \end{equation*}
        One can compute the Fourier transform of $\Psi_v$ explicitly. Indeed, we have
        \begin{align*}
            \hat{\Psi_v}(t, \xi) &= (2\pi)^{-\frac{d}{2}}\int_{\R^d} e^{-ix \cdot \xi} e^{i\frac{|x|^2}{4t}} \vartheta\left(t^{-\frac{1}{2}}(x-2vt)\right)\dx \\
            &= (2\pi)^{-\frac{d}{2}}\int_{\R^d} e^{-i(x-2vt+2vt)\cdot\xi} e^{i \frac{|x-2vt+2vt|^2}{4t}}\vartheta(t^{-\frac{1}{2}}(x-2vt))\dx\\
            &=
                (2\pi)^{-\frac{d}{2}}e^{-2ivt\cdot \xi}\int_{\R^d}e^{-i(x-2vt)\cdot\xi} e^{i\frac{|x-2vt|^2}{4t}+i(x-2vt)\cdot v+ it|v|^2}\vartheta(t^{-\frac{1}{2}}(x-2vt))\dx
            \\
            &= 
                (2\pi)^{-\frac{d}{2}}e^{-it|\xi|^2}e^{it|\xi - v|^2} \int_{\R^d}e^{-i(x-2vt)\cdot(\xi - v)}e^{i\frac{|x-2vt|^2}{4t}}\vartheta(t^{-\frac{1}{2}}(x-2vt))\dx
            \\
            &= t^{\frac{d}{2}}e^{-it|\xi|^2} \tilde{\vartheta}(t^{\frac{1}{2}}(\xi - v)),
        \end{align*}
        where the function $\tilde{\vartheta}(\xi)$ is given by 
        \begin{equation*}
            \tilde{\vartheta}(\xi) = e^{i|\xi|^2}\F\left[e^{i\frac{|x|^2}{4}}\vartheta(x)\right](\xi).
        \end{equation*}
        Direct computation shows that in fact 
        \begin{equation*}
            \int_{\R^d} \tilde{\vartheta}(\xi) \dxi = \int_{\R^d} \vartheta(x) \dx =1. 
        \end{equation*}
        If we remember where we started, we had 
        \begin{align*}
            \gamma(t, v) &= \int_{\R^d} u(t, x) \bbar{\Psi_v}(t,x) \dx \\
            &= \int_{\R^d}t^{\frac{d}{2}}\hat{u}(t, \xi) \bbar{e^{-it|\xi|^2} \tilde{\vartheta}(t^{\frac{1}{2}}(\xi - v))}\dxi.
        \end{align*}
        Using the fact that the complex conjugate in the definition of $\gamma$ introduces a minus sign in the argument of $\hat{\Psi_v}$, we see  
        \begin{equation}\label{E:gammafreqrepn}
            \gamma(t, \xi) = e^{it|\xi|^2} \hat{u}(t, \xi) \ast_{\xi} t^{\frac{d}{2}}\tilde{\vartheta}(t^{\frac{1}{2}}\xi).
        \end{equation}
        We'd now like to compare $\gamma(t, v)$ to a solution $u(t, x)$ of \eqref{E:VNLS} along a ray $\Gamma_v$. Since we'll need it, note that by direct computation, we have the following equality in the sense of Fourier multipliers: 
        \begin{equation}
            |\nabla_v|^s = (2t)^s |\nabla_x|^s \qtq{for all} s \in \R.
        \end{equation}
        where we recall that $v = \tfrac{x}{2t}$ (in particular, the units are consistent across the equals sign). With these definitions in hand, we can state the lemma.
        \begin{lemma}\label{L:gammabounds}
            The function $\gamma(t, v)$ satisfies the bounds 
            \begin{equation}
                \|\gamma\|_{L^\infty} \lesssim t^{\frac{d}{2}}\|u\|_{L^\infty}, \quad \|\gamma\|_{L_v^2} \lesssim \|u\|_{L_x^2},\quad \| |\nabla_v|^\beta \gamma\|_{L_v^2} \lesssim\| |J|^{\beta} u\|_{L_x^2}.
            \end{equation}
            We also have the physical space bounds
            \[
                |u(t, 2vt) - t^{-\frac{d}{2}} e^{i\frac{|x|^2}{4t}} \gamma(t, v)| \lesssim t^{-\frac{\beta}{2} - \frac{d}{4}} \| |J|^{\beta} u\|_{L_x^2}
            \]  
            and the Fourier space bounds 
            \[
                |\hat{u}(t, \xi) - e^{-it|\xi|^2} \gamma(t, \xi)| \lesssim t^{\frac{d}{4}- \frac{\beta}{2}} \| |J|^{\beta}u\|_{L_x^2}.
            \]

        \end{lemma}
        \begin{proof}
            Let $w = e^{-i|x|^2/4t}u$. Then we can express $\gamma$ in terms of $w$ as a convolution with respect to the variable $v$: 
            \begin{equation}\label{E:gammaconvrep}
                t^{-\frac{d}{2}}\gamma(t, v) = w(t, 2vt) \ast_v 2^d t^{\frac{d}{2}} \vartheta(2t^\frac{1}{2}v),
            \end{equation}
            noting that $v \mapsto 2^d t^{\frac{d}{2}}\vartheta((2t)^{\frac{1}{2}}v)$ has integral 1. By Young's inequality, then, we have the immediate convolution bounds 
            \begin{align*}
                \|\gamma(t, v)\|_{L^\infty} &\lesssim t^{\frac{d}{2}}\|w(t, 2vt)\|_{L^\infty} = t^{\frac{d}{2}}\|u\|_{L^\infty}, \\
                \|\gamma(t, v)\|_{L_v^2} &\lesssim t^{\frac{d}{2}}\|w(t, 2vt)\|_{L_v^2} \sim \|u\|_{L_x^2}.
            \end{align*}
            A straightforward estimate via Young's inequality and the commutativity of Fourier multipliers yields the third bound, \textit{viz}
            \begin{align*}
                \| |\nabla_v|^{\beta} \gamma(t, v)\|_{L_v^2} &\lesssim \| |\nabla_v|^{\beta} w(t, 2vt)\|_{L_v^2} \\
                &\lesssim (2t)^{\beta-d} \left\|\int_{\R^d} e^{i\frac{\eta}{2t}\cdot x}\left|\frac{\eta}{2t}\right|^\beta \tilde{w}\left(t,\frac{\eta}{2t}\right)\d\eta \right\|_{L_x^2} \\
                &\lesssim (2t)^{\beta} \left\|\int_{\R^d} e^{i \xi \cdot x} |\xi|^{\beta} \hat{w}(t, \xi) \dxi \right\|_{L_x^2} \\
                &\lesssim  \| |J|^{\beta}u\|_{L_x^2},  
            \end{align*}
            where to get from the second line to the third we used the fact that $v = \tfrac{x}{2t}$ implies that their Fourier dual variables satisfy $\eta = 2t\xi$. To get rid of the $(2t)^\beta$, we used the definition \eqref{E:Jtdefn1} to absorb it back into the $|\nabla_x|^{\beta}$. \par
            For the physical space bounds, we need to compare $u(t, 2vt)$ with $\gamma(t, v)$. To do this, note that 
            \begin{equation}
                |u(t, 2vt) - t^{-\frac{d}{2}}e^{i\frac{|x|^2}{4t}}\gamma(t, v)| = |w(t, 2vt) - t^{-\frac{d}{2}}\gamma(t, v)|. 
            \end{equation}
            We can then use the representation \eqref{E:gammaconvrep} to write 
            \begin{equation}
                |w(t, 2vt) -t^{-\frac{d}{2}}\gamma(t, v)| = \left| \int_{\R^d} w(t, 2(v-z)t) (2t^{\frac{1}{2}})^d\vartheta((2t)^{\frac{1}{2}}z) \d z - w(t, 2vt)\right|.
            \end{equation}
            Using that $\vartheta(z)$ has unit integral, we can again rewrite the above equation as 
            \begin{equation}\label{E:differencerepn}
                |w(t, 2vt) -t^{-\frac{d}{2}}\gamma(t, v)| = \left| \int_{\R^d} [w(t, 2(v-z)t) - w(t, 2vt)](2t^{\frac{1}{2}})^d\vartheta((2t)^{\frac{1}{2}}z) \d z\right|.
            \end{equation}
            Then using homogeneous Sobolev embedding, we find (using that $\tfrac{d}{2}< \beta < 1+\tfrac{d}{2}$), 
            \begin{equation}
                |w(t, 2t(v-z)) - w(t, 2vt)| \lesssim |z|^{\beta - \frac{d}{2}} \| |\nabla_v|^{\beta} w(t, 2vt)\|_{L_v^2}.
            \end{equation}
            This allows us to continue the estimate of \eqref{E:differencerepn} by 
            \begin{align*}
                |w(t, 2vt) - t^{-\frac{d}{2}}\gamma(t, v)| &\lesssim \int_{\R^d} |z|^{\beta - \frac{d}{2}} \| |\nabla_v|^{\beta} w(t, 2vt)\|_{L_v^2} (2t^{\frac{1}{2}})^d |\vartheta((2t)^{\frac{1}{2}}z)| \d z \\
                &\lesssim \| |\nabla_v|^\beta w(t, 2vt)\|_{L_v^2} (2t^{\frac{1}{2}})^{-(\beta-\frac{d}{2})} \\
                &\lesssim t^{-\frac{\beta}{2}-\frac{d}{4}} \| |J|^{\beta}u\|_{L_x^2}. 
            \end{align*}
             
            To obtain the frequency estimate, we recall the definition \eqref{E:gammafreqrepn} and that $\tilde{\vartheta}$ has unit integral to write 
            \begin{equation}
                |\hat{u}(t, \xi) - e^{-it |\xi|^2} \gamma(t, \xi)| \leq \int_{\R^d} |  e^{it |\xi-\eta|^2}\hat{u}(\xi - \eta) - e^{it|\xi|^2}\hat{u}(\xi)| (2t)^{\frac{d}{2}} |\tilde{\vartheta}((2t)^{\frac{1}{2}}\eta)| \d\eta.
            \end{equation}
            Applying the same Sobolev embedding argument from earlier, we can bound the difference in the integral above by 
            \begin{equation}
                |e^{it|\xi - \eta|^2}\hat{u}(\xi - \eta) - e^{it|\xi|^2} \hat{u}(\xi)| \lesssim |\eta|^{\beta - \frac{d}{2}} \| |\nabla_\xi|^\beta (e^{it |\xi|^2} \hat{u})\|_{L_\xi^2}.
            \end{equation}
            Substituting this into the integral and applying the same argument as for the physical-space case yields the estimate 
            \begin{equation}
                |\hat{u}(t, \xi) - e^{-it|\xi|^2} \gamma(t, \xi)| \lesssim t^{\frac{d}{4}- \frac{\beta}{2}} \| |J|^{\beta}u\|_{L_x^2}.
            \end{equation}
        \end{proof}
        
        \begin{lemma}
            Let $u(t, x)$ be a solution to \eqref{E:VNLS}. Then we have
            \begin{equation}\label{E:gammaODE}
                \partial_t \gamma(t, v) = \frac{i}{2t} (|\cdot|^{-1} \ast |\gamma(t, v)|^2)\gamma(t, v) + \cR(t, v),
            \end{equation}
            where the remainder $\cR(t,v)$ satisfies the estimate 
            \begin{multline}\label{E:remainderestimate}
                \|\cR(t, v)\|_{L^\infty} \lesssim t^{-\frac{3}{2}}\|\jbrak{J}^\beta u\|_{L^2}+ {t^{-\frac{d}{2}+\frac{3}{4}-\frac{\beta}{2d}}\|u\|_{L^\infty}^{1+\frac{1}{d}} \|u\|_{L^2}^{\frac{2(d-1)}{d}} \| |J|^\beta u\|_{L^2}^\frac{1}{d}} \\+t^{\frac{d}{2}-\frac{1}{d}(\frac{\beta}{2}+\frac{d}{4})}\|u\|_{L^\infty}^{1+\frac{1}{d}}\|u\|_{L^2}^{\frac{2(d-1)}{d}} \| |J|^\beta u\|_{L^2}^{\frac{1}{d}}.
            \end{multline}
        \end{lemma}    
        \begin{proof}
            Begin by taking the time derivative of $\gamma(t, v)$. By definition, this yields 
            \begin{align*}
                \partial_t \gamma(t, v) &= \int_{\R^d} \partial_t u \bbar{\Psi_v} + u \bbar{\partial_t \Psi_v} \dx \\
                &= \int_{\R^d} i \left[\Delta u - (|\cdot|^{-1} \ast |u|^2) u\right] \bbar{\Psi_v} + u \bbar{\partial_t \Psi_v} \dx \\
                &= i\int_{\R^d} u \bbar{\left[i \partial_t + \Delta\right]\Psi_v} \dx - i \int_{\R^d} (|\cdot|^{-1} \ast |u|^2)u \bbar{\Psi_v} \dx,
            \end{align*}
            where the integrations by parts to move the Laplace operator from the $u$ to the wavepacket $\Psi_v$ are justified using the decay of $\Psi_v$. Next, we use the computation \eqref{E:approxsoln} to write 
            \begin{equation}
                i \int_{\R^d} u \bbar{\left[i\partial_t + \Delta\right]\Psi_v} \dx = \frac{i}{2t} \int_{\R^d} u \M(-t) \nabla \cdot \left\{-i(x-2vt)\bbar{\vartheta} + 2t^{\frac{1}{2}}\bbar{\nabla\vartheta}\right\}\dx.
            \end{equation}
            We can then integrate by parts to rewrite the integral as 
            \begin{multline*}
                \frac{i}{2t} \int_{\R^d} u \M(-t) \nabla \cdot \left\{-it^{-\frac{1}{2}}(x-2vt)\bbar{\vartheta} + 2t^{\frac{1}{2}}\bbar{\nabla\vartheta}\right\}\dx \\= -\frac{i}{2t}\int_{\R^d} \nabla (\M(-t) u) \cdot \left(-i (x-2vt)\bbar{\vartheta} + 2t^{\frac{1}{2}} \bbar{\nabla \vartheta} \right)\dx.
            \end{multline*}
            We then use the definition of $J$ to rewrite this once more as 
                \begin{multline}
                    -\frac{i}{2t}\int_{\R^d} \nabla (e^{-i\frac{|x|^2}{4t}} u) \cdot \left(-i (x-2vt)\bbar{\vartheta} + 2t^{\frac{1}{2}} \bbar{\nabla \vartheta} \right)\dx \\= -\frac{1}{(2t)^2} \int_{\R^d} e^{-i\frac{|x|^2}{4t}} Ju \cdot \left(i (x-2vt)\bbar{\vartheta} -2t^\frac{1}{2}\bbar{\nabla \vartheta}\right) \dx.
                \end{multline}
            Combining this with the second term above, we can rewrite things in the following fashion: 
            \begin{align}
                \partial_t \gamma(t, v) &= -\frac{1}{(2t)^2}\int_{\R^d} \M(-t) Ju \cdot \{-i (x-2vt) \bbar{\vartheta} +2t^{\frac{1}{2}}\bbar{\nabla \vartheta}\} \dx \\
                &-i \int_{\R^d} u\bbar{\Psi_v} (|\cdot|^{-1} \ast (|u|^2 - |u(t, 2vt)|^2)) \dx \\
                &+ \begin{multlined}[t]
                    i \gamma |\cdot|^{-1} \ast |u(t, vt)|^2 - i \frac{1}{2t} (|\cdot|^{-1} \ast |\gamma(t, v)|^2 )\gamma \\+ i \frac{1}{2t}(|\cdot|^{-1} \ast |\gamma(t, v)|^2)\gamma
                \end{multlined} \\
                &\defe \frac{i}{2t}(|\cdot|^{-1} \ast |\gamma(t, v)|^2 )\gamma + \mathcal{R}_1 + \mathcal{R}_2 + \mathcal{R}_3,
            \end{align}
            where we define 
            \begin{equation}
                \cR_1 \defe -\frac{1}{(2t)^2}\int_{\R^d} \M(-t) Ju \cdot \{-i (x-2vt) \bbar{\vartheta} +2t^{\frac{1}{2}}\bbar{\nabla \vartheta}\} \dx,
            \end{equation}
            \begin{equation}
               \cR_2 \defe -i \int_{\R^d} u\bbar{\Psi_v} (|\cdot|^{-1} \ast (|u|^2 - |u(t, 2vt)|^2)) \dx, 
            \end{equation} 
            and 
            \begin{equation}
                \cR_3 \defe -\gamma(t, v)[|\cdot|^{-1} \ast (|u(t, 2vt)|^2 - \frac{1}{2t} |\gamma(t, v)|^2)],
            \end{equation}
            and where here and throughout we abide by the convention in \Cref{R:convolution}.
            We can now proceed to estimate each of the terms $\|\cR_j\|_{L^\infty}$. For $\cR_1$, we can change variables $x = 2tz$ to rewrite $\cR_1$ as a convolution: 
            \begin{align}
                \cR_1 &=
                \begin{multlined}[t]
                    -(2t)^{d-2}\int_{\R^d} \tilde{\M}(-t) Ju(2tz) \\\times \{-2it(z-v)\bbar{\vartheta(2t^{\frac{1}{2}}(z-v))} + 2t^{\frac{1}{2}} \bbar{\nabla \vartheta(2t^{\frac{1}{2}}(z-v))}\}\dz 
                \end{multlined}\\
                &= -\frac{1}{(2t)^2} \tilde{\M}(-t) (Ju)(2tz) \ast_v \left\{-2itz \vartheta(2t^{\frac{1}{2}}z) + 2t^{\frac{1}{2}} \bbar{\nabla \vartheta(2t^{\frac{1}{2}}z)}\right\}.
            \end{align}
            Here we write $\tilde{\M}(-t) = e^{-it|z|^2}$. Using Young's convolution inequality, we can place the first part of the convolution in $L^\infty$ and the second in $L^1$ to obtain 
            \begin{align}
                \|\cR_1\|_{L^\infty} &\lesssim t^{d-2} \|Ju\|_{L^\infty} t^{\frac{1}{2}-\frac d2} \\
                &\lesssim t^{\frac{d}{2}-\frac{3}{2}} \|e^{it\Delta} x e^{-it\Delta}u\|_{L^\infty}\\
                &\lesssim t^{- \frac{3}{2}}\|x e^{-it\Delta}u \|_{L^1} \\
                &\lesssim t^{-\frac{3}{2}}\|\jbrak{x}^\beta e^{-it\Delta}u \|_{L^2} \\
                &\lesssim t^{-\frac{3}{2}}\{\|u\|_{L^2} + \||J|^\beta u\|_{L^2}\}\\
                &\lesssim t^{-\frac{3}{2}}\|\jbrak{J}^\beta u\|_{L^2}
            \end{align}
            which is an acceptable estimate. 
            For $\cR_2$, we have 
            \begin{equation}
                \cR_2 = -i\int_{\R^d} u \bbar{\Psi}_v \left[|\cdot|^{-1} \ast(|u|^2(x) - |u|^2(2tv))\right] \dx.
            \end{equation}
            We thus have (using the same notation for $w(t,x)$ as above)
            \begin{align}
                |\cR_2| &\lesssim \|u\|_{L^\infty} \int_{\R^d} \bbar{\vartheta}\left(\frac{x-2tv}{t^{\frac{1}{2}}}\right)\left[|\cdot|^{-1} \ast (|w|^2(x) - |w|^2(2tv))\right] \dx \\
                &\lesssim \|u\|_{L^\infty} t^{\frac{d}{2}}\int_{\R^d} 2^d t^{\frac{d}{2}}\bbar{\vartheta}(2t^{\frac{1}{2}}z)\left[ |\cdot|^{-1} \ast (|w|^2(2t(v-z)) - |w|^2(2tv)) \right] \dz \\
                &\lesssim \|u\|_{L^\infty}t^\frac{d}{2}\int_{\R^d} 2^d t^{\frac{d}{2}} \bbar{\vartheta}(2t^{\frac{1}{2}}z) \left\| |\cdot|^{-1} \ast (|w|^2(2t(v-z)) - |w|^2(2tv))\right\|_{L_v^\infty} \dz. \label{E:R2Linftyest}
            \end{align}
            Now we need to estimate the convolution in $L^\infty$. To do this, we use \Cref{L:interplemma}, which yields
            \begin{equation*}
                \||\cdot|^{-1} \ast (|w|^2(2t(v-z)) - |w|^2(2tv))\|_{L_v^\infty}\lesssim \| |w|^2(2t(v-z)) - |w|^2(2tv)\|_{L^{\frac{d}{d-1},1}}.
            \end{equation*}
            The right-hand side of this equation is bounded by 
            \begin{equation*}
            | |w|^2(2t(v-z)) - |w|^2(2tv)\|_{L^{\frac{d}{d-1},1}} \lesssim \begin{multlined}[t]
                \|w(2t(v-z)) - w(2tv)\|_{L_v^2}^{\frac{d-1}{d}} \\ \times \|w(2t(v-z)) - w(2tv)\|_{L_v^\infty}^{\frac{1}{d}} \\
                   \times \| w(2t(v-z)) + w(2tv)\|_{L_v^2}^{\frac{d-1}{d}} \\ \times \|w(2t(v-z)) + w(2tv)\|_{L_v^\infty}^{\frac{1}{d}}.
            \end{multlined}
            \end{equation*}
            We bound each term individually. We crudely control the $L^2$ norms by 
            \begin{equation*}
                \|w(2t(v-z)) \pm w(2tv)\|_{L_v^2}^{\frac{d-1}{d}} \lesssim t^{-\frac{d-1}{2}}\|u\|_{L^2}^{\frac{d-1}{d}}
            \end{equation*}
            the $L^\infty$ norm of the sum by 
            \begin{equation*}
                \|w(2t(v-z)) + w(2tv) \|_{L_v^\infty}^{\frac{1}{d}} \lesssim \|u\|_{L^\infty}^{\frac{1}{d}},
            \end{equation*}
            and we use Sobolev embedding to control 
            \begin{equation*}
                \|w(2t(v-z)) - w(2tv)\|_{L_v^\infty}^{\frac{1}{d}} \lesssim |z|^{\frac{\beta}{d}-\frac{1}{2}}\||J|^\beta u\|_{L^2}^{\frac{1}{d}} t^{-\frac{1}{2}}.
            \end{equation*}
            Putting this all together, we see that we can estimate the right-hand side of  \eqref{E:R2Linftyest} by 
            \begin{align*}
                \mathrm{RHS}\eqref{E:R2Linftyest}&\lesssim \|u\|_{L^\infty}^{1+\frac{1}{d}} \|u\|_{L^2}^{\frac{2(d-1)}{d}} \| |J|^\beta u\|_{L^2}^\frac{1}{d}t^{\frac{d}{2}-d+1-\frac{1}{2}} \int_{\R^d} 2^d t^{\frac{d}{2}}\bbar{\vartheta}(2t^{\frac{1}{2}}z) |z|^{\frac{\beta}{d}-\frac{1}{2}} \dz \\
                &\lesssim \|u\|_{L^\infty}^{1+\frac{1}{d}} \|u\|_{L^2}^{\frac{2(d-1)}{d}} \| |J|^\beta u\|_{L^2}^\frac{1}{d} t^{-\frac{d}{2}+\frac{3}{4}-\frac{\beta}{2d}}
            \end{align*}
            which is an acceptable estimate.\par
            Finally, we need an estimate on $\cR_3$. We have 
            \begin{equation}
                \cR_3 = -\gamma(t, v) \left[|\cdot|^{-1} \ast \left(|u(t, 2vt)|^2 - \frac{1}{2t}|\gamma(t, v)|^2\right)\right].
            \end{equation}
            Now notice that we can write 
            \begin{align}
                |\cdot|^{-1} \ast |u(t, 2vt)|^2 &= \int_{\R^d} \frac{|u(t, y)|^2}{|2vt- y|} \dy \\
                &= (2t)^{d-1} \int_{\R^d} \frac{|u(t, 2tz)|^2}{|v - z|} \dz.
            \end{align}
            Using this, $\cR_3$ can be rewritten as 
            \begin{equation}
                \cR_3 = -(2t)^{d-1}\gamma(t, v) \int_{\R^d} \frac{|u(2tz)|^2 - (2t)^{-d} |\gamma(t,z)|^2}{|v-z|}\dz.
            \end{equation}
            Now using the $L^\infty$ endpoint of the Hardy-Littlewood-Sobolev inequality \cite[Theorem 2.6]{oneilConvolutionOperatorsSpaces1963}, we can estimate the integral by 
            \begin{align}
                \left|\int_{\R^d} \frac{|w(2tz)|^2 - (2t)^{-d}|\gamma(t, z)|^2}{|v-z|}\dz\right| &\lesssim \left\| |w(2tz)|^2 - (2t)^{-d}|\gamma(t, z)|^2\right\| _{L^{\frac{d}{d-1},1}} \\
                &\lesssim \begin{multlined}[t]
                    \left\|w(2tz) - (2t)^{-\frac{d}{2}}\gamma(t, z)\right\|_{L^{\frac{2d}{d-1},2}} \\\times\left\|w(2tz) + (2t)^{-\frac{d}{2}}\gamma(t, z)\right\|_{L^{\frac{2d}{d-1}, 2}}.
                \end{multlined}
            \end{align}
            We now need to estimate each of the Lorentz norms above. By \Cref{L:interplemma}, we find that 
            \begin{align}
                &\begin{multlined}
                    \left\| w(2tz) - (2t)^{-\frac{d}{2}}\gamma(t, z)\right\|_{L^{\frac{2d}{d-1}, 2}} \\\lesssim \|w(2tz) - (2t)^{-\frac{d}{2}}\gamma(t, z)\|_{L^2}^{\frac{d-1}{d}}\left\|w(2tz) - (2t)^{-\frac{d}{2}}\gamma(t, z)\right\|_{L^\infty}^{\frac{1}{d}}\label{E:minusterm}
                \end{multlined}
                \\
                &\begin{multlined}
                    \left\|w(2tz) + (2t)^{-\frac{d}{2}}\gamma(t, z)\right\|_{L^{\frac{2d}{d-1}, 2}} \\\lesssim \left\|w(2tz) + (2t)^{-\frac{d}{2}}\gamma(t, z)\right\|_{L^2}^{\frac{d-1}{d}}\left\|w(2tz) + (2t)^{-\frac{d}{2}}\gamma(t, z)\right\|_{L^\infty}^{\frac{1}{d}}.\label{E:plusterm}
                \end{multlined}
            \end{align}
            We now estimate each of the norms on the right-hand side individually. \par
            For the terms in \eqref{E:plusterm}, we use \Cref{L:gammabounds} and the triangle inequality: 
            \begin{align}
                &\|w(2tz)+(2t)^{-\frac{d}{2}}\gamma(t, z)\|_{L^2}^{\frac{d-1}{d}} \lesssim t^{-\frac{d-1}{2}}\|u\|_{L^2}^{\frac{d-1}{d}}, \\
                &\|w(2tz) + (2t)^{-\frac{d}{2}}\gamma(t, z)\|_{L^\infty}^{\frac{1}{d}} \lesssim \|u\|_{L^\infty}^{\frac{1}{d}}.
            \end{align}
            For the terms in \eqref{E:minusterm}, we crudely estimate the difference in $L^2$ by the sum, and apply the analysis from \eqref{E:plusterm}. For the $L^\infty$ term, though, we use the physical space estimate from \Cref{L:gammabounds}. This nets us the pair of estimates 
            \begin{align}
                &\|w(2tz) - (2t)^{-\frac{d}{2}}\gamma(t, z)\|_{L^2}^{\frac{d-1}{d}} \lesssim t^{-\frac{d-1}{2}}\|u\|_{L^2}^{\frac{d-1}{d}}, \\
                &\|w(2tz) - (2t)^{-\frac{d}{2}}\gamma(t, z)\|_{L^\infty}^{\frac{1}{d}} \lesssim t^{-\frac{1}{d}(\frac{\beta}{2}+\frac{d}{4})}\| |J|^\beta u\|_{L^2}^{\frac{1}{d}}.
            \end{align}
            Note that the physical space estimate can be reproven with an extra factor of $2$; all that changes is that in \eqref{E:differencerepn} we have a $2^{d/2}$ that we can freely replace with $2^d$, using the fact that $d > 0$.
            
            Putting all this information together, we establish the final estimate 
            \begin{equation}
                    \|\cR_3\|_{L^\infty} \lesssim t^{\frac{d}{2}-\frac{1}{d}(\frac{\beta}{2}+\frac{d}{4})}\|u\|_{L^2}^{\frac{2(d-1)}{d}} \|u\|_{L^\infty}^{1+\frac{1}{d}}\| |J|^\beta u\|_{L^2}^{\frac{1}{d}},
            \end{equation}
            which is acceptable. 
        \end{proof} 

        Having written down the approximate ODE for the wavepacket $\gamma(t, v)$, we would like to establish the bounds \eqref{E:sharpdecay} and \eqref{E:energygrowth} claimed as part of \Cref{T:maintheorem}. This will follow from the following 
        \begin{lemma}\label{L:bootstraparg}
           Let $0 < \eps \ll 1$ and suppose that  $\|u_0\|_{H^{0, \beta}} = \eps$. Let $u(t, x)$ be the corresponding global solution to \eqref{E:VNLS} given by \Cref{T:GWP}. Then the bounds \eqref{E:sharpdecay} and \eqref{E:energygrowth} hold globally in time.
        \end{lemma}
        \begin{proof}
            To prove the lemma, we will run a bootstrap argument under the assumption
            \begin{equation}
            \|u(t)\|_{L^\infty} \leq \eps^{\frac{1}{2}}t^{-\frac{d}{2}},
        \end{equation}
        where $\eps \ll 1$ is the size of the initial datum $u_0$ in the $H^{0, \beta}$ norm. We will prove the bound \eqref{E:energygrowth} first, using Gr\"onwall's inequality, then use the ODE for $\gamma(t, v)$ to prove the bound \eqref{E:sharpdecay}, which will close the bootstrap under the hypothesis $\eps \ll 1$. 
        
        To begin, by the global well-posedness result \cref{T:GWP}, we know that we have 
        \begin{equation}
            \| |J|^\beta u(1)\|_{L^2} \lesssim \| |x|^{\beta}u_0\|_{L^2} < \eps. 
        \end{equation}
        To move past time $t = 1$, we write the Duhamel formula started at $t = 1$: 
        \begin{equation}
            u(t) = e^{i(t-1)\Delta}u_1 -i \int_1^t e^{i(t-s)\Delta} [(|\cdot|^{-1} \ast |u|^2)u] \ds.
        \end{equation}
        Applying $|J|^\beta$ to the equation and estimating in $L^2$, the triangle inequality implies that 
        \begin{equation}
            \| |J|^\beta u\|_{L^2} \lesssim \eps + \int_1^t \| |J|^\beta [(|\cdot|^{-1} \ast |u|^2)u]\|_{L^2} \ds.
        \end{equation}
        By using the identity \eqref{E:jbetaidentity}, we see that it suffices to understand 
        \begin{equation}
            \| |\nabla|^\beta[|\cdot|^{-1} \ast |u|^2]u\|_{L^2}.
        \end{equation}
        In the end, we will replace $u \mapsto w$ and add back the factor $(2it)^\beta$ appearing in \eqref{E:jbetaidentity}; this will give the result for $|J|^\beta u$. 
        Using the fractional product rule, we can control this expression by 
        \begin{equation}
            \| |\cdot|^{-1} \ast |\nabla|^\beta |u|^2\|_{L^p} \|u\|_{L^q} + \| |\cdot|^{-1} \ast |u|^2 \|_{L^\infty} \| |\nabla|^\beta u\|_{L^2},
        \end{equation}
        where $\frac{1}{p}+\frac{1}{q}= \frac{1}{2}$. For the $L^\infty$ term, we use the endpoint Hardy-Littlewood-Sobolev inequality to control the $L^\infty$ norm by 
        \begin{equation}
            \||\cdot|^{-1} \ast |u|^2\|_{L^\infty} \lesssim \|u\|_{L^{\frac{2d}{d-1}, 2}}^2.
        \end{equation}
        Using \Cref{L:interplemma}, we conclude an estimate of the form
        \begin{equation}
            \|u\|_{L^{\frac{2d}{d-1},2}}^2 \lesssim \|u\|_{L^\infty}^{\frac{2}{d}}\|u\|_{L^2}^{\frac{2(d-1)}{d}},
        \end{equation}
        which is acceptable in view of our bootstrap hypothesis. For the other term, we use Hardy-Littlewood-Sobolev and the fractional product rule to estimate 
        \begin{equation}
            \| |\cdot|^{-1} \ast |\nabla|^\beta |u|^2\|_{L^p} \lesssim \| |\nabla|^\beta |u|^2\|_{L^{\tilde{p}}} \lesssim \||\nabla|^\beta u\|_{L^2} \|u\|_{L^r},
        \end{equation}
        where we have the conditions 
        \begin{equation}
            \frac{1}{p}+ \frac{d-1}{d} = \frac{1}{\tilde{p}} \qtq{and} \frac{1}{2}+ \frac{1}{r}=\frac{1}{\tilde{p}}.
        \end{equation}
        In light of the conditions, we see that both $r$ and $q$ are larger than $2$; we can thus interpolate again between $L^2$ and $L^\infty$ and use the conditions on $p, q, r$ to deduce
        \begin{equation}
            \| |\cdot|^{-1} \ast |\nabla|^\beta u\|_{L^p} \|u\|_{L^q} \lesssim \| |\nabla|^\beta u\|_{L^2} \|u\|_{L^2}^{\frac{2(d-1)}{d}}\|u\|_{L^\infty}^\frac{2}{d},
        \end{equation}
        which is acceptable. Putting this all together, we find that 
        \begin{equation}
            \| |J|^\beta u(t)\|_{L^2} \lesssim 2\eps + \tilde{C}\eps^2\int_1^t \frac{1}{s} \| |J|^\beta u(s)\|_{L^2} \ds,
        \end{equation}
        for some constant $\tilde{C} \ll C$. Hence by the Gr\"onwall inequality
        \begin{equation}
            \| |J|^\beta u\|_{L^2} \lesssim 2\eps\jbrak{t}^{\eps^{3-\frac{1}{d}}}.
        \end{equation}
        By conservation of mass, this further implies that 
        \begin{equation}
            \| \jbrak{J}^\beta u\|_{L^2} \lesssim 2\eps\jbrak{t}^{\eps^{3-\frac{1}{d}}}.
        \end{equation}
        Next, we want to close the bootstrap and establish the required pointwise decay bound. Using Lemma \ref{L:gammabounds}, we know that 
        \begin{equation}
            \|e^{-i\frac{|x|^2}{4t}}u(t, x) - t^{-\frac{d}{2}}\gamma\|_{L_x^\infty} \lesssim t^{-\frac{\beta}{2}-\frac{d}{4}} \| |J|^\beta u\|_{L_x^2}.
        \end{equation}
        In light of the estimate on $\| |J|^\beta u\|_{L^2}$, we know that 
        \begin{equation}
            \| e^{-i\frac{|x|^2}{4t}}u - t^{-\frac{d}{2}}\gamma\|_{L^\infty} \lesssim 2\eps\jbrak{t}^{-\frac{\beta}{2}-\frac{d}{4}+\eps^{3-\frac{1}{d}}}.
        \end{equation}
        It thus remains to estimate $\gamma$; these estimates will help establish the bound \eqref{E:sharpdecay}. First, by the local theory and \Cref{L:gammabounds}, we have 
        \begin{equation}
            |\gamma(1, v)|\lesssim 2\eps.
        \end{equation}
        To close the argument, we introduce the integrating factor 
        \begin{equation}
            B(t) = \exp\left(-\int_{1}^{t}\frac{i}{2s}|\cdot|^{-1}\ast |\gamma(s, v)|^2 \ds\right).
        \end{equation}
        Noting that $B$ has modulus $1$ and multiplying through in \eqref{E:gammaODE} by $B$, we see that in modulus 
        \begin{equation}
            |\gamma(t, v)| \lesssim |\gamma(1, v)| + \int_1^t |\cR(s, v) |\ds. 
        \end{equation}
        Under the bootstrap hypotheses, we know that $\cR$ is integrable; indeed, we have a bound on $\cR$ of the form
        \begin{equation}
            |\cR(s, v)|\lesssim \eps(1+\eps^{2-\frac{1}{d}}) s^{-1-\delta + \eps^{3-\frac{1}{d}}},
        \end{equation}
        where $\delta$ = $\beta - \frac{d}{2}$ is positive; for example, $\delta = \eps$ suffices. Integrating, then, we have a bound on $|\gamma(t, v)|$ by
        \begin{equation}
            |\gamma(t, v)|\lesssim \eps(1+\eps^{2-\frac{1}{d}}).
        \end{equation}
        Thus we have 
        \begin{equation}
            \|u\|_{L^\infty} \lesssim 2\eps |t|^{-\frac{d}{2}},
        \end{equation}
        under the constraint that $\eps \ll 1$, which completes the argument. 
        \end{proof}
        
        \section{Asymptotic Behavior of Solutions}\label{S:asymptotics}
        In this section, we will write down an explicit asymptotic expansion for solutions to \eqref{E:VNLS}. Recall earlier that we defined the quantity  
        \begin{equation}
            B(t) = \exp\left(-\int_1^t \frac{i}{2s}|\cdot|^{-1}\ast |\gamma(s, v)|^2 \ds\right).
        \end{equation}
        as an integrating factor for \eqref{E:gammaODE}. We thus set 
        \begin{equation}
            G(t) = B(t)\gamma(t, v).            
        \end{equation}
        Using the bootstrap argument from earlier, we have the estimate 
        \begin{equation}
            \|\partial_t G\|_{L^\infty} \lesssim \eps^3 t^{-1-\eps+\eps^{3-\frac{1}{d}}}.
        \end{equation}
        Using the fundamental theorem of calculus and completeness of $L^\infty$, we find that there exists $\mathscr{W}_0 \in L^\infty$ such that 
        \begin{equation}
            \|G(t) - \mathscr{W}_0 \|_{L^\infty}\to 0  \qtq{as} t \to \infty.
        \end{equation}
        In particular, since $B$ has modulus $1$, we know that $|\gamma(t, v)| \to |\sW_0|(v)$ in $L^\infty$. In particular, we know that 
        \begin{equation}
            B(t) = \exp\left(-\frac{i}{2}(|\cdot|^{-1} \ast |\sW_0|^2(v))\log(t)-i\Phi(t)\right),
        \end{equation}
        where $\Phi(t)$ has a limit $\Phi_\infty$ in $L^\infty$.In particular, if we set $\sW = e^{-i\Phi_\infty} \sW_0$, we find the expansion 
        \begin{equation}
            \gamma(t, v) = e^{-\frac{i}{2}(|\cdot|^{-1} \ast |\sW(v)|^2)\log(t)}\sW(v) + \mathcal{O}(t^{-\eps})
        \end{equation}
        for any $\eps > 0$. This completes the argument, as the asymptotic expansion for $u$ then immediately follows from \eqref{L:gammabounds}, using the fact that $\||J|^\beta u\|_{L^2}$ grows like $2\eps\jbrak{t}^{\eps^{3-\frac{1}{d}}}$. 
        \appendix
        \section{Interpolation and Lorentz Spaces} \label{AS:InterpAppendix}
            In this appendix, we will give the details of the claim in \Cref{L:interplemma}. To begin, \cite[Theorem 5.3.1]{berghInterpolationSpacesIntroduction1976} is as follows:
            \begin{theorem*}
                Suppose that $0 < p_0, p_1, q_0, q_1 \leq \infty$ and write 
                \begin{equation*}
                    \frac{1}{p} = \frac{1-\theta}{p_0} + \frac{\theta}{p_1} \qtq{where} 0 < \theta < 1.
                \end{equation*}
                Then if $p_0 \neq p_1$, we have 
                \begin{equation*}
                    [L^{p_0, q_0}, L^{p_1, q_1}]_{\theta, q} = L^{p, q}.
                \end{equation*}
                If $p_0 = p_1 = p$, this result holds provided that 
                \begin{equation*}
                    \frac{1}{q} = \frac{1-\theta}{q_0} + \frac{\theta}{q_1}.
                \end{equation*}
            \end{theorem*}
            Applying this theorem with $(p_0, q_0) = (2, 2)$, $(p_1, q_1) = (\infty, \infty)$ and $(p, q) = \left( \frac{2d}{d-1}, 2\right)$ yields the first claim of \Cref{L:interplemma}, as the resulting value for $\theta$ is $\theta = \frac{1}{d} \in (0, 1)$, so that the above Theorem applies. 

            To conclude the remainder of \Cref{L:interplemma} involving interpolation of norms, we note that by definition the space $L^{p, q}$ in the Theorem above is defined by real interpolation. Using \cite[Theorem 3.5.1]{berghInterpolationSpacesIntroduction1976}, we see that $L^{\frac{2d}{d-1}, 2}$ (in the notation of \cite{berghInterpolationSpacesIntroduction1976}) is a space of class $\mathscr{C}\left(\frac{1}{d}, (L^2, L^\infty)\right)$. Using \cite[Item (b), p.49]{berghInterpolationSpacesIntroduction1976} and the fact that being class $\mathscr{C}\left(\frac{1}{d}, (L^2, L^\infty)\right)$ implies being class $\mathscr{C}_J\left(\frac{1}{d}, (L^2, L^\infty)\right)$, we conclude that 
            \begin{equation*}
                \|u\|_{L^{\frac{2d}{d-1}, 2}} \lesssim \|u\|_{L^2}^{\frac{d-1}{d}}\|u\|_{L^\infty}^{\frac{1}{d}},
            \end{equation*}
            which is what we claimed.
        \nocite{*}
    \bibliography{testing-by-wavepackets}
    \bibliographystyle{bjoern_style}
\end{document}